\documentclass{article}
\usepackage{amsfonts}
\usepackage{amsmath}

\newtheorem{theorem}{Theorem}

\newtheorem{definition}[theorem]{Definition}

\newtheorem{lemma}[theorem]{Lemma}
\newtheorem{notation}[theorem]{Notation}

\newenvironment{proof}[1][Proof]{\noindent\textbf{#1.} }{\ \rule{0.5em}{0.5em}}
\input{tcilatex}
\begin{document}

\title{Taylor estimate for differential equations driven by $\Pi $-rough
paths}
\author{Danyu Yang}
\maketitle

\begin{abstract}
We obtain a remainder estimate for the truncated Taylor expansion for
differential equations driven by weakly geometric $\Pi $-rough paths for $%
\Pi =\left( p_{1},\cdots ,p_{k}\right) $, $p_{i}\geq 1$. When there exists $%
p\geq 1$ such that $p_{i}=pk_{i}^{-1}\ $for some $k_{i}\in \left\{ 1,\dots ,%
\left[ p\right] \right\} $, we obtain a refined Taylor remainder estimate
that contains a factorial decay component. The remainder estimates are in
the right order as they are comparable to the next term in the Taylor
expansion.
\end{abstract}

\section{Introduction}

Consider the differential equation%
\begin{equation*}
dy_{t}=\sum_{i=1}^{d}f_{i}\left( y_{t}\right) dx_{t}^{i},y_{0}=\xi ,
\end{equation*}%
where the driving path $x$ can be highly oscillating. In the seminal paper 
\cite{lyons1998differential}, Lyons builds the theory of rough paths. By
lifting $x$ to a group-valued path, the theory of rough paths establishes
the existence and uniqueness of the solution and demonstrates the continuity
of the solution with respect to the driving signal in rough path metric.
Successful applications of the theory include differential equations driven
by general stochastic processes, stochastic flow, large deviation principle,
densities for rough differential equations \cite{friz2010multidimensional}
etc.

Based on modified Euler estimates, Davie \cite{davie2008differential}
develops an alternative approach to rough paths theory when $p<3$ and
constructs examples to illustrate the sharpness of the results. By making
systematic use of sub-Riemannian geodesics, Friz and Victoir \cite%
{friz2008euler} obtain Euler estimates for rough differential equations when 
$p\geq 1$ and obtain strong remainder estimates for stochastic Taylor
expansions.

The driving rough path can be inhomogeneous \cite{lejay2006p,
friz2010multidimensional, gyurko2016differential}. The special and important
case of $\left( p,q\right) $-rough paths is studied by Lejay and Victoir 
\cite{lejay2006p} and by Friz and Victoir \cite[Chapter 12]%
{friz2010multidimensional}. By adapting Lyons' approach \cite%
{lyons1998differential}, Gyurk\'{o} \cite{gyurko2016differential} proves the
existence, uniqueness and continuity of the solution to differential
equations driven by geometric $\Pi $-rough paths (inhomogeneous geometric
rough paths).

Building on the well-posedness result by Gyurk\'{o} \cite%
{gyurko2016differential}, in Lemma \ref{Lemma Euler estimate} we obtain a
Taylor remainder estimate for differential equations driven by weakly
geometric $\Pi $-rough paths for $\Pi =\left( p_{1},\cdots ,p_{d}\right) $, $%
p_{i}\geq 1$. Apart from $p_{i}\geq 1$, there is no further restriction on $%
p_{i}$. The proof of Lemma \ref{Lemma Euler estimate} relies on an
inhomogeneous counterpart of the sub-Riemannian geodesic methodology
developed by Friz and Victoir \cite{friz2008euler} and the analytical
approach introduced by Davie \cite{davie2008differential}.

In the special case when there exists $p\geq 1$ such that $p_{i}=pk_{i}^{-1}$
for some $k_{i}\in \left\{ 1,\dots ,\left[ p\right] \right\} $, $i=1,2,\dots
,d$, we obtain in Theorem \ref{Theorem Taylor estimate} a refined Taylor
remainder estimate which contains a factorial decay component. Theorem \ref%
{Theorem Taylor estimate} can be applied to differential equations driven by
general stochastic processes with a drift term of finite $p\left[ p\right]
^{-1}$-variation almost surely. Theorem \ref{Theorem Taylor estimate} can
also be useful when dealing with differential equations driven by branched
rough paths and quasi-geometric rough paths which are isomorphic to weakly
geometric $\Pi $-rough paths \cite{boedihardjo2019isomorphism,
yang2022remainder, bellingeri2020quasi}. The proof of Theorem \ref{Theorem
Taylor estimate} is based on the neo-classical inequality \cite%
{lyons1998differential, hara2010fractional} and the backward mathematical
induction developed in \cite{boedihardjo2015uniform}.

Although a geometric $\Pi $-rough path can be viewed as a geometric $p$%
-rough path for sufficiently large $p$, doing so imposes unnecessary
requirements on the vector field. In this paper, the sub-Riemannian
geometrical tools and the analytical methods are carefully adapted to the
inhomogeneous setting to accommodate the reduced regularity of the vector
field. The remainder estimates in Lemma \ref{Lemma Euler estimate} and
Theorem \ref{Theorem Taylor estimate} are in the right order as they are
comparable to the next term in the Taylor expansion.

\section{Notations and Results}

\begin{notation}
Fix real numbers $p_{i}\geq 1$, $i=1,\dots ,d$. Denote $\Pi :=\left(
p_{1},\cdots ,p_{d}\right) $ and denote $p_{\max }:=\max_{i=1}^{d}p_{i}$.
\end{notation}

\begin{notation}
Let $W$ denote the set of finite sequences of elements in $\left\{ 1,2,\dots
,d\right\} $ including the empty sequence denoted as $\varepsilon $. Set $%
\left\vert \varepsilon \right\vert :=0$. For $w=i_{1}\cdots i_{m}\in W$,
denote 
\begin{equation*}
\left\vert w\right\vert :=\frac{1}{p_{i_{1}}}+\cdots +\frac{1}{p_{i_{m}}}.
\end{equation*}
\end{notation}

\begin{notation}
Let $G_{\Pi }$ denote the set of mappings $a:\left\{ w\in W|\left\vert
w\right\vert \leq 1\right\} \rightarrow 
\mathbb{R}
$ that satisfy%
\begin{equation*}
\left( a,w_{1}\right) \left( a,w_{2}\right) =\left( a,sh\left( w_{1}\otimes
w_{2}\right) \right)
\end{equation*}%
for every $w_{1},w_{2}\in W$, $\left\vert w_{1}\right\vert +\left\vert
w_{2}\right\vert \leq 1$, where the shuffle product $sh$ is defined in \cite[%
Section 1.4]{reutenauer1993free}.
\end{notation}

For $a,b\in G_{\Pi }$ and $w\in W,\left\vert w\right\vert \leq 1$, define $%
\left( ab,w\right) =\sum_{uv=w}\left( a,u\right) \left( b,v\right) $ with $%
uv $ denoting the concatenation of $u$ with $v$. Then $G_{\Pi }$ is a group 
\cite[Theorem 3.2, Corollary 3.3]{reutenauer1993free}.

\begin{notation}
\label{Notation control}For $X:\left[ 0,T\right] \rightarrow G_{\Pi }$,
denote $X_{s,t}:=X_{s}^{-1}X_{t}$ and define%
\begin{equation*}
\omega _{X}\left( s,t\right) :=\sum_{w\in W,0<\left\vert w\right\vert \leq
1}\left( \sup_{s=t_{0}<t_{1}<\cdots <t_{n}=t,n\geq
1}\sum_{k=0}^{n-1}\left\vert \left( X_{t_{k},t_{k+1}},w\right) \right\vert ^{%
\frac{1}{\left\vert w\right\vert }}\right) \text{.}
\end{equation*}
\end{notation}

\begin{notation}
$X:\left[ 0,T\right] \rightarrow G_{\Pi }$ is a weakly geometric $\Pi $%
-rough path if $X\ $is continuous and $\omega _{X}\left( 0,T\right) <\infty $%
.
\end{notation}

Suppose Lipschitz functions and norms are defined as in \cite[Definition
1.2.4]{lyons1998differential}.

\begin{notation}
\label{Notation Xbar}Suppose $X:\left[ 0,T\right] \rightarrow G_{\Pi }$ and $%
f_{i}:%
\mathbb{R}
^{n}\rightarrow 
\mathbb{R}
^{n}$ is $Lip\left( \gamma _{i}\right) $, $i=1,2,\dots ,d$. Define $\bar{X}:%
\left[ 0,T\right] \rightarrow G_{\Pi }$ as%
\begin{equation*}
\left( \bar{X}_{t},i_{1}\cdots i_{k}\right) =\left\Vert f_{i_{1}}\right\Vert
_{Lip\left( \gamma _{i_{1}}\right) }\cdots \left\Vert f_{i_{k}}\right\Vert
_{Lip\left( \gamma _{i_{k}}\right) }\left( X_{t},i_{1}\cdots i_{k}\right) .
\end{equation*}
\end{notation}

For $f,g,\varphi :%
\mathbb{R}
^{n}\rightarrow 
\mathbb{R}
^{n}$, denote $f\left( \varphi \right) :=\left( D\varphi \right) \!f$ and
define $\left( f\circ g\right) \left( \varphi \right) :=f\left( g\left(
\varphi \right) \right) $. Let $I$ denote the identity function on $%
\mathbb{R}
^{n}$.

\begin{notation}
For sufficiently smooth $f_{i}:%
\mathbb{R}
^{n}\rightarrow 
\mathbb{R}
^{n}$, $i=1,2,\dots ,d$ and $w=i_{1}\cdots i_{m}\in W$, denote $%
F^{w}:=\left( f_{i_{1}}\circ \cdots \circ f_{i_{m}}\right) \left( I\right) $%
. For the empty sequence $\varepsilon $, denote $F^{\varepsilon }:=I$.
\end{notation}

For $\bar{X}$ in Notation \ref{Notation Xbar}, recall $\omega _{\bar{X}}$ in
Notation \ref{Notation control}.

\begin{lemma}
\label{Lemma Euler estimate}Let $X$ be a weakly geometric $\Pi $-rough path
for $\Pi =\left( p_{1},\dots ,p_{d}\right) $, $p_{i}\geq 1$, and let $f_{i}:%
\mathbb{R}
^{n}\rightarrow 
\mathbb{R}
^{n}$ be $Lip\left( \gamma _{i}\right) $ for $\gamma _{i}>p_{\max }\left(
1-p_{i}^{-1}\right) +1,i=1,\dots ,d$. Let $y:\left[ 0,T\right] \rightarrow 
\mathbb{R}
^{n}$ denote the unique path-level solution of the rough differential
equation: 
\begin{equation*}
dy_{t}=\sum_{i=1}^{d}f_{i}\left( y_{t}\right) dX_{t}^{i},y_{0}=\xi \in 
\mathbb{R}
^{n}.
\end{equation*}%
Denote $\theta :=\min_{w\in W,\left\vert w\right\vert >1}\left\vert
w\right\vert $. Then there exists a constant $C$ that only depends on $%
p_{1},\cdots ,p_{d},d$ such that%
\begin{equation*}
\left\vert y_{t}-y_{s}-\sum_{w\in W,0<\left\vert w\right\vert \leq
1}F^{w}\left( y_{s}\right) \left( X_{s,t},w\right) \right\vert \leq C\left(
\omega _{\bar{X}}\left( s,t\right) ^{\theta }\wedge \omega _{\bar{X}}\left(
s,t\right) \right)
\end{equation*}%
for every $\,0\leq s\leq t\leq T$.
\end{lemma}

The existence and uniqueness of the solution $y$ follows from \cite[Theorem
4.3]{gyurko2016differential}\cite[Lemma 3.5]{boedihardjo2019isomorphism}.
The Lipschitz condition on $f_{i}$ is encoded in the assumption of \cite[%
Theorem 4.3]{gyurko2016differential} that $h_{2}$ is $\limfunc{Lip}^{\Gamma
_{2},\Pi _{2}}$. In fact $\gamma _{i}-1=p_{\max }\bar{\gamma}_{i}$ where $%
\bar{\gamma}_{i}$ denotes the $\gamma _{i}$ in \cite[Theorem 4.3]%
{gyurko2016differential}. Lemma \ref{Lemma Euler estimate} serves as an
intermediate step to proving Theorem \ref{Theorem Taylor estimate}.

Let $\left[ p\right] $ denote the largest integer that is less or equal to $%
p $, and let $\lfloor \gamma \rfloor $ denote the largest integer that is
strictly less than $\gamma $.

\begin{theorem}
\label{Theorem Taylor estimate}For $p\geq 1$, let $X$ be a weakly geometric $%
\Pi $-rough path for $\Pi =\left( p_{1},\dots ,p_{d}\right) $ with $%
p_{i}=pk_{i}^{-1}$ for some $k_{i}\in \left\{ 1,\dots ,\left[ p\right]
\right\} $. Suppose $f_{i}:%
\mathbb{R}
^{n}\rightarrow 
\mathbb{R}
^{n}$ is $Lip\left( \gamma _{i}\right) $ for $\gamma _{i}>p_{\max }\left(
1-k_{i}p^{-1}\right) +1$. Let $y$ denote the unique path-level solution of
the rough differential equation%
\begin{equation*}
dy_{t}=\sum_{i=1}^{d}f_{i}\left( y_{t}\right) dX_{t}^{i},y_{0}=\xi \in 
\mathbb{R}
^{n}.
\end{equation*}%
Denote $\gamma :=\min_{i=1}^{d}\left\{ \left( \gamma _{i}-1\right) pp_{\max
}^{-1}+k_{i}\right\} $ and denote $N:=\lfloor \gamma \rfloor $. Then $N\geq %
\left[ p\right] $ and there exists a constant $C$ that only depends on $%
p_{1},\cdots ,p_{d},d$ such that%
\begin{equation}
\left\vert y_{t}-y_{s}-\sum_{w\in W,0<\left\vert w\right\vert \leq \frac{N}{p%
}}F^{w}\left( y_{s}\right) \left( X_{s,t},w\right) \right\vert \leq
CN!d^{N+1}\beta ^{N}\frac{\omega _{\bar{X}}\left( s,t\right) ^{\frac{N+1}{p}}%
}{\left( \frac{N+1}{p}\right) !}  \label{main estimate}
\end{equation}%
for every $\,0\leq s\leq t\leq T$, where%
\begin{equation*}
\beta :=p\left( 1+\sum_{k=2}^{\infty }\left( 2k^{-1}\right) ^{\frac{\left[ p%
\right] +1}{p}}\right) .
\end{equation*}
\end{theorem}

The estimate $\left( \ref{main estimate}\right) $ captures correctly the
order of the next term in the Taylor expansion: 
\begin{equation}
\sum_{w\in W,\left\vert w\right\vert =\frac{N+1}{p}}F^{w}\left( y_{s}\right)
\left( X_{s,t},w\right) .  \label{magnitude}
\end{equation}

1. The factor $N!$ in $\left( \ref{main estimate}\right) $ comes from
iterated composition of the vector field. Denote $f^{\circ \left( k+1\right)
}:=D\!\left( f^{\circ k}\right) \!f$ with $f^{\circ 1}:=f$. Let $f\left(
y\right) =e^{-y}\ $and consider the differential equation $%
dy_{t}=e^{-y_{t}}dt$, $y_{0}=0$. Then $\left\Vert f\right\Vert _{Lip\left(
k\right) }=1$ on $y\geq 0$ for $k=1,2,\dots $ and $\left\vert f^{\circ
\left( N+1\right) }\left( 0\right) \right\vert =N!$.

2. When $p_{i}=p$ ($X$ is homogeneous), there are total $d^{N+1}$ terms in $%
\left( \ref{magnitude}\right) $ which explains the factor $d^{N+1}$ in $%
\left( \ref{main estimate}\right) $.

3. Based on the neo-classical inequality \cite{hara2010fractional} and an
inhomogeneous analogue of \cite[Theorem 2.2.1]{lyons1998differential}, for $%
w\in W$, $\left\vert w\right\vert =\frac{N+1}{p}$, 
\begin{equation*}
\left\vert \left( \bar{X}_{s,t},w\right) \right\vert \leq \beta ^{N}\frac{%
\omega _{\bar{X}}\left( s,t\right) ^{\frac{N+1}{p}}}{\left( \frac{N+1}{p}%
\right) !}
\end{equation*}%
which explains the estimate $\left( \ref{main estimate}\right) $.

\section{Proofs}

Replace $X$ by $\bar{X}$ in Notation \ref{Notation Xbar} and replace $f_{i}$
by $\left\Vert f_{i}\right\Vert _{Lip\left( \gamma _{i}\right) }^{-1}f_{i}$
so that $\left\Vert f_{i}\right\Vert _{Lip\left( \gamma _{i}\right) }=1$.
The differential equation stays unchanged. Denote $\omega \left( s,t\right)
:=\omega _{\bar{X}}\left( s,t\right) $ as in Notation \ref{Notation control}%
. Constants in proofs may depend on $p_{1},\cdots ,p_{d},d$. The exact value
of constants may change.

\begin{notation}
For a continuous bounded variation path $x:\left[ 0,T\right] \rightarrow 
\mathbb{R}
^{d}$, denote $S_{\Pi }\left( x\right) :\left\{ \left( s,t\right) |0\leq
s\leq t\leq T\right\} \rightarrow G_{\Pi }$ as%
\begin{equation*}
\left( S_{\Pi }\left( x\right) _{s,t},i_{1}\cdots i_{m}\right)
=\idotsint\limits_{s<u_{1}<\cdots <u_{m}<t}dx_{u_{1}}^{i_{1}}\cdots
dx_{u_{m}}^{i_{m}}
\end{equation*}%
for $i_{1}\cdots i_{m}\in W$, $\left\vert i_{1}\cdots i_{m}\right\vert \leq
1 $.
\end{notation}

\begin{definition}
For $a\in G_{\Pi }$, define%
\begin{equation*}
\left\Vert a\right\Vert ^{\prime }:=\inf_{x}\sum_{i=1}^{d}\left\Vert
x^{i}\right\Vert _{1-var,\left[ 0,1\right] }^{\frac{p_{i}}{p_{\max }}}
\end{equation*}%
where the infimum is taken over all continuous bounded variation paths $x:%
\left[ 0,1\right] \rightarrow 
\mathbb{R}
^{d}$ that satisfy%
\begin{equation*}
\left( S_{\Pi }\left( x\right) _{0,1},w\right) =\left( a,w\right) \text{ for
every }w\in W,\left\vert w\right\vert \leq 1\text{.}
\end{equation*}%
The infimum can be achieved at a continuous bounded variation path which is
called a \emph{geodesic} associated with $a\in G_{\Pi }$.
\end{definition}

The existence of geodesic can be proved similarly to \cite[Theorem 7.32]%
{friz2010multidimensional}. Since $p_{i}p_{\max }^{-1}\leq 1$, the
sub-additivity and continuity of $\left\Vert \cdot \right\Vert ^{\prime }$
can be proved similarly to \cite[Proposition 7.40]{friz2010multidimensional}%
. For $a\in G_{\Pi }$, define%
\begin{equation*}
\left\Vert a\right\Vert :=\sum_{w\in W,0<\left\vert w\right\vert \leq
1}\left\vert \left( a,w\right) \right\vert ^{\frac{1}{p_{\max }\left\vert
w\right\vert }}.
\end{equation*}%
For $\lambda >0$ and $a\in G_{\Pi }$, define $\delta _{\lambda }a\in G_{\Pi
} $ as 
\begin{equation*}
\left( \delta _{\lambda }a,w\right) :=\lambda ^{p_{\max }\left\vert
w\right\vert }\left( a,w\right) \text{ for every }w\in W,\left\vert
w\right\vert \leq 1\text{.}
\end{equation*}%
Both $\left\Vert \cdot \right\Vert ^{\prime }$ and $\left\Vert \cdot
\right\Vert $ are continuous homogeneous norms with respect to $\delta
_{\lambda }$. By arguments similar to \cite[Theorem 7.44]%
{friz2010multidimensional}, $\left\Vert \cdot \right\Vert ^{\prime }$ and $%
\left\Vert \cdot \right\Vert $ are equivalent up to a constant depending on $%
p_{1},\cdots ,p_{d},d$.

\begin{lemma}
\label{Lemma control of geodesic}Suppose $X:\left[ 0,T\right] \rightarrow
G_{\Pi }\ $is a weakly geometric $\Pi $-rough path. For $0\leq s<t\leq T$,
let $x^{s,t}=\left( x^{s,t,1},\dots ,x^{s,t,d}\right) :\left[ s,t\right]
\rightarrow 
\mathbb{R}
^{d}$ be a geodesic associated with $X_{s,t}$. Then there exists a constant $%
C$ depending on $p_{1},\cdots ,p_{d},d$ such that for $i=1,\dots ,d$,%
\begin{equation*}
\left\Vert x^{s,t,i}\right\Vert _{1-var,\left[ s,t\right] }\leq C\omega
\left( s,t\right) ^{\frac{1}{p_{i}}}.
\end{equation*}
\end{lemma}

\begin{proof}
Since $\left\Vert \cdot \right\Vert ^{\prime }$ and $\left\Vert \cdot
\right\Vert $ are equivalent, we have%
\begin{equation*}
\left\Vert x^{s,t,i}\right\Vert _{1-var,\left[ s,t\right] }^{\frac{p_{i}}{%
p_{\max }}}\leq \left\Vert X_{s,t}\right\Vert ^{\prime }\leq C\left\Vert
X_{s,t}\right\Vert \leq C\omega \left( s,t\right) ^{\frac{1}{p_{\max }}}.
\end{equation*}
\end{proof}

Recall Notation $F^{i_{1}\cdots i_{m}}:=\left( f_{i_{1}}\circ \cdots \circ
f_{i_{m}}\right) \left( I\right) $.

\begin{lemma}
\label{Lemma indices}For some $\alpha >0$, suppose $f_{i}$ is $Lip\left(
\gamma _{i}\right) $ for $\gamma _{i}>p_{\max }\left( \alpha
-p_{i}^{-1}\right) +1$, $i=1,2,\dots ,d$. Then for $w\in W$, $\left\vert
w\right\vert \leq \alpha $, $F^{w}$ is $Lip\left( \eta \right) $ for some $%
\eta >1$.
\end{lemma}

\begin{proof}
Fix $w=i_{1}\cdots i_{m}\in W$ that satisfies $\left\vert w\right\vert \leq
\alpha $. Suppose $\gamma _{i_{j}}\leq j$ for some $j\in \left\{ 1,\dots
,m\right\} $. Then%
\begin{equation*}
\left\vert w\right\vert =\frac{1}{p_{i_{1}}}+\cdots +\frac{1}{p_{i_{m}}}\geq 
\frac{j-1}{p_{\max }}+\frac{1}{p_{i_{j}}}\geq \frac{\gamma _{i_{j}}-1}{%
p_{\max }}+\frac{1}{p_{i_{j}}}>\alpha ,
\end{equation*}%
contradicts with $\left\vert w\right\vert \leq \alpha $. Hence $\gamma
_{i_{j}}>j$ for every $j\in \left\{ 1,\dots ,m\right\} $ and $F^{w}$ is $%
Lip\left( \eta \right) $ for some $\eta >1$.
\end{proof}

\begin{lemma}
\label{Lemma bound on Fw inhomogeneous}Suppose $f_{i}$ is $Lip\left( \gamma
_{i}\right) $ for $\gamma _{i}>p_{\max }\left( 1-p_{i}^{-1}\right) +1$, $%
i=1,2,\dots ,d$. Then for $w\in W,\left\vert w\right\vert \leq 1$ and $%
i_{0}\in \left\{ 1,2,\dots ,d\right\} $, $\sup_{y\in 
\mathbb{R}
^{n}}\left\vert F^{i_{0}w}\left( y\right) \right\vert \leq \left[ p_{\max }%
\right] !$.
\end{lemma}

\begin{proof}
Based on Lemma \ref{Lemma indices} (with $\alpha =1$), when $\left\vert
w\right\vert \leq 1$, $F^{w}$ is $Lip\left( \eta \right) $ for some $\eta >1$%
. Suppose $w=i_{1}\cdots i_{m}$. Then $F^{i_{0}w}=\left( f_{i_{0}}\circ
f_{i_{1}}\circ \cdots \circ f_{i_{m}}\right) \left( I\right) $ contains $m!$
functions as the differential operator $f_{i_{j}}D$ choosing one from $%
\{f_{i_{j+1}},\dots ,f_{i_{m}}\}$ for $j=0,\dots ,m-1$. Since $\left\Vert
f_{i}\right\Vert _{Lip\left( \gamma _{i}\right) }=1$, each of these
functions is bounded by $1$. As a result, $\sup_{y\in 
\mathbb{R}
^{n}}\left\vert F^{i_{0}w}\left( y\right) \right\vert \leq m!\leq \left[
p_{\max }\right] !$.
\end{proof}

\begin{lemma}
\label{Lemma A}Let $x=\left( x^{1},\dots ,x^{d}\right) :\left[ 0,1\right]
\rightarrow 
\mathbb{R}
^{d}$ be a continuous bounded variation path. Suppose $f_{i}:%
\mathbb{R}
^{n}\rightarrow 
\mathbb{R}
^{n}$ is $Lip\left( \gamma _{i}\right) $ for $\gamma _{i}>p_{\max }\left(
1-p_{i}^{-1}\right) +1$, $i=1,2,\dots ,d$. Let $y:\left[ 0,1\right]
\rightarrow 
\mathbb{R}
^{n}$ be the unique solution of the ordinary differential equation%
\begin{equation*}
dy_{t}=\sum_{i=1}^{d}f_{i}\left( y_{t}\right) dx_{t}^{i},y_{0}=\xi .
\end{equation*}%
Suppose there exist $C>0$ and $K\in \left[ 0,1\right] $ such that $%
\left\Vert x^{i}\right\Vert _{1-var,\left[ 0,1\right] }\leq CK^{\frac{1}{%
p_{i}}}$ for $i=1,2,\dots ,d$. Denote $\theta :=\min_{w\in W,\left\vert
w\right\vert >1}\left\vert w\right\vert $. Then there exists a constant $%
C^{\prime }$ depending on $C,p_{1},\cdots ,p_{d},d\ $such that 
\begin{equation*}
\left\vert y_{1}-\xi -\sum_{w\in W,\left\vert w\right\vert \leq
1}F^{w}\left( \xi \right) \left( S_{\Pi }\left( x\right) _{0,1},w\right)
\right\vert \leq C^{\prime }K^{\theta }.
\end{equation*}
\end{lemma}

\begin{proof}
By the fundamental theorem of calculus and Lemma \ref{Lemma bound on Fw
inhomogeneous}, 
\begin{eqnarray*}
&&\left\vert y_{1}-\xi -\sum_{w\in W,\left\vert w\right\vert \leq
1}F^{w}\left( \xi \right) \left( S_{\Pi }\left( x\right) _{0,1},w\right)
\right\vert \\
&\leq &\sum_{\substack{ \left\vert i_{1}\cdots i_{m}\right\vert \leq 1  \\ %
\left\vert i_{0}i_{1}\cdots i_{m}\right\vert >1}}\;\;\idotsint%
\limits_{0<u_{0}<\cdots <u_{m}<1}\left\vert F^{i_{0}\cdots i_{m}}\left(
y_{u_{0}}\right) \right\vert \left\vert dx_{u_{0}}^{i_{0}}\right\vert \cdots
\left\vert dx_{u_{m}}^{i_{m}}\right\vert \\
&\leq &C^{\prime }K^{\theta }
\end{eqnarray*}%
where we used $\left\Vert x^{i}\right\Vert _{1-var,\left[ 0,1\right] }\leq
CK^{\frac{1}{p_{i}}}\ $for some $K\in \left[ 0,1\right] $.
\end{proof}

\begin{lemma}
\label{Lemma B}Let $x=\left( x^{1},\dots ,x^{d}\right) :\left[ 0,1\right]
\rightarrow 
\mathbb{R}
^{d}$ be a continuous bounded variation path. Suppose $f=\left( f_{1},\dots
,f_{d}\right) $ satisfies that $f_{i}:%
\mathbb{R}
^{n}\rightarrow 
\mathbb{R}
^{n}$ is $Lip\left( \gamma _{i}\right) $ for $\gamma _{i}>p_{\max }\left(
1-p_{i}^{-1}\right) +1$, $i=1,2,\dots ,d$. Let $y^{j}:\left[ 0,1\right]
\rightarrow 
\mathbb{R}
^{n},j=1,2\ $be the solution of the ordinary differential equation%
\begin{equation*}
dy_{t}^{j}=\sum_{i=1}^{d}f_{i}\left( y_{t}^{j}\right)
dx_{t}^{i},y_{0}^{j}=\xi ^{j}.
\end{equation*}%
Suppose there exist $C>0$ and $K\in \left[ 0,1\right] $ such that $%
\left\Vert x^{i}\right\Vert _{1-var,\left[ 0,1\right] }\leq CK^{\frac{1}{%
p_{i}}}$ for $i=1,2,\dots ,d$. Then there exists a constant $C^{\prime }\ $%
depending on $C,d$ such that%
\begin{equation*}
\left\vert y_{1}^{1}-\xi ^{1}-\left( y_{1}^{2}-\xi ^{2}\right) \right\vert
\leq C^{\prime }\left\vert \xi ^{1}-\xi ^{2}\right\vert K^{\frac{1}{p_{\max }%
}}.
\end{equation*}
\end{lemma}

\begin{proof}
Denote $\overline{y}_{t}:=y_{t}^{1}-y_{t}^{2}$. Then%
\begin{eqnarray*}
&&\left\vert \overline{y}_{t}-\overline{y}_{0}\right\vert \\
&\leq &\sum_{i=1}^{d}\left\Vert f_{i}\right\Vert _{Lip\left( \gamma
_{i}\right) }\left\vert \overline{y}_{0}\right\vert \left\Vert
x^{i}\right\Vert _{1-var,\left[ 0,1\right] }+\sum_{i=1}^{d}\int_{0}^{t}\left%
\Vert f_{i}\right\Vert _{Lip\left( \gamma _{i}\right) }\left\vert \overline{y%
}_{r}-\overline{y}_{0}\right\vert \left\vert dx_{r}^{i}\right\vert .
\end{eqnarray*}%
Since $\left\Vert f_{i}\right\Vert _{Lip\left( \gamma _{i}\right) }=1$ and $%
\left\Vert x^{i}\right\Vert _{1-var,\left[ 0,1\right] }\leq CK^{\frac{1}{%
p_{i}}}$ for $K\in \left[ 0,1\right] $, applying Gronwall's Lemma \cite[%
Lemma 3.2]{friz2010multidimensional}, the proposed estimate holds.
\end{proof}

\begin{proof}[Proof of Lemma \protect\ref{Lemma Euler estimate}]
Firstly suppose $x:\left[ 0,T\right] \rightarrow 
\mathbb{R}
^{d}$ is continuous and of bounded variation. Let $y:\left[ 0,T\right]
\rightarrow 
\mathbb{R}
^{n}$ denote the solution of the ODE%
\begin{equation*}
dy_{t}=f\left( y_{t}\right) dx_{t},y_{0}=\xi .
\end{equation*}%
For $\left[ s,t\right] \subset \left[ 0,T\right] $, let $x^{s,t}:\left[ s,t%
\right] \rightarrow 
\mathbb{R}
^{d}$ be a geodesic associated with $S_{\Pi }\left( x\right) _{s,t}$, and
let $y^{s,t}:\left[ s,t\right] \rightarrow 
\mathbb{R}
^{n}$ denote the solution of the ODE%
\begin{equation*}
dy_{r}^{s,t}=f\left( y_{r}^{s,t}\right) dx_{r}^{s,t},y_{s}^{s,t}=y_{s}.
\end{equation*}%
Denote%
\begin{equation*}
\Gamma ^{s,t}:=y_{t}-y_{s}-\left( y_{t}^{s,t}-y_{s}\right) .
\end{equation*}

Let $\omega $ be the control associated with $S_{\Pi }\left( x\right) $ as
in Notation \ref{Notation control}. Firstly suppose $\omega \left(
s,t\right) \leq 1$. For $s\leq u\leq t$, let $x^{s,u,t}:\left[ s,t\right]
\rightarrow 
\mathbb{R}
^{d}$ denote the concatenation of $x^{s,u}$ with $x^{u,t}$, and let $%
y^{s,u,t}:\left[ s,t\right] \rightarrow 
\mathbb{R}
^{n}$ denote the solution of the ODE%
\begin{equation*}
dy_{r}^{s,u,t}=f\left( y_{r}^{s,u,t}\right)
dx_{r}^{s,u,t},y_{s}^{s,u,t}=y_{s}.
\end{equation*}%
Based on Lemma \ref{Lemma control of geodesic},%
\begin{eqnarray*}
\left\Vert x^{s,t,i}\right\Vert _{1-var,\left[ s,t\right] } &\leq &C\omega
\left( s,t\right) ^{\frac{1}{p_{i}}}. \\
\left\Vert x^{s,u,t,i}\right\Vert _{1-var,\left[ s,t\right] } &\leq
&\left\Vert x^{s,u,i}\right\Vert _{1-var,\left[ s,u\right] }+\left\Vert
x^{u,t,i}\right\Vert _{1-var,\left[ u,t\right] }\leq 2C\omega \left(
s,t\right) ^{\frac{1}{p_{i}}}.
\end{eqnarray*}%
Then based on Lemma \ref{Lemma A} and Lemma \ref{Lemma B}, 
\begin{eqnarray*}
&&\left\vert \Gamma ^{s,u}+\Gamma ^{u,t}-\Gamma ^{s,t}\right\vert \\
&\leq &\left\vert y_{t}^{s,t}-y_{t}^{s,u,t}\right\vert +\left\vert
y_{t}^{s,u,t}-y_{u}^{s,u}-\left( y_{t}^{u,t}-y_{u}\right) \right\vert \\
&\leq &C_{1}\omega \left( s,t\right) ^{\theta }+C_{2}\left\vert \Gamma
^{s,u}\right\vert \omega \left( u,t\right) ^{\frac{1}{p_{\max }}}
\end{eqnarray*}%
Then%
\begin{equation}
\left\vert \Gamma ^{s,t}\right\vert \leq \left( 1+C_{2}\omega \left(
s,t\right) ^{\frac{1}{p_{\max }}}\right) \left( \left\vert \Gamma
^{s,u}\right\vert +\left\vert \Gamma ^{u,t}\right\vert \right) +C_{1}\omega
\left( s,t\right) ^{\theta }.  \label{inductive step}
\end{equation}%
With $C_{1}$ and $C_{2}$ in $\left( \ref{inductive step}\right) $, denote%
\begin{equation*}
\delta :=\max \left( \left( C_{2}\right) ^{p_{\max }},\left( C_{1}\right) ^{%
\frac{1}{\theta }}\right) \omega \left( s,t\right)
\end{equation*}%
Denote $\left[ t_{0}^{0},t_{1}^{0}\right] :=\left[ s,t\right] $. Divide $%
\left[ t_{j}^{n},t_{j+1}^{n}\right] =\left[ t_{2j}^{n+1},t_{2j+1}^{n+1}%
\right] \cup \left[ t_{2j+1}^{n+1},t_{2j+2}^{n+1}\right] $ such that 
\begin{equation*}
\omega \left( t_{2j}^{n+1},t_{2j+1}^{n+1}\right) =\omega \left(
t_{2j+1}^{n+1},t_{2j+2}^{n+1}\right) \leq 2^{-1}\omega \left(
t_{j}^{n},t_{j+1}^{n}\right)
\end{equation*}%
for $j=0,\dots ,2^{n}-1,n=0,1,2,\dots $, which is possible because $\omega $
is super-additive and continuous \cite[Proposition 5.8]%
{friz2010multidimensional}. By iteratively applying $\left( \ref{inductive
step}\right) $, as $\omega \left( s,t\right) \leq 1$,%
\begin{eqnarray}
\left\vert \Gamma ^{s,t}\right\vert &\leq &\overline{\lim_{n\rightarrow
\infty }}\left( 1+\sum_{j=0}^{n}\prod\limits_{k=0}^{j}\left( 1+\left( \frac{%
\delta }{2^{k}}\right) ^{\frac{1}{p_{\max }}}\right) \frac{1}{2^{\left(
j+1\right) \left( \theta -1\right) }}\right) \delta ^{\theta }  \notag \\
&&+\overline{\lim_{n\rightarrow \infty }}\prod\limits_{k=0}^{n}\left(
1+\left( \frac{\delta }{2^{k}}\right) ^{\frac{1}{p_{\max }}}\right) \left(
\sum_{j=0}^{2^{n+1}-1}\left\vert \Gamma ^{t_{j}^{n},t_{j+1}^{n}}\right\vert
\right)  \notag \\
&\leq &\exp \left( \frac{2^{\frac{1}{p_{\max }}}C}{2^{\frac{1}{p_{\max }}}-1}%
\right) \left( \frac{2^{\theta -1}}{2^{\theta -1}-1}\delta ^{\theta }+%
\overline{\lim_{n\rightarrow \infty }}\sum_{j=0}^{2^{n}-1}\left\vert \Gamma
^{t_{j}^{n},t_{j+1}^{n}}\right\vert \right) .
\label{inner estimate Gamma s,t}
\end{eqnarray}

Since $x$ is continuous with bounded variation and $x^{s,t}$ is a geodesic
associated with $S_{\Pi }\left( x\right) _{s,t}$, we have $S_{\Pi }\left(
x^{s,t}\right) _{s,t}=S_{\Pi }\left( x\right) _{s,t}$ and $\left\Vert
x^{s,t}\right\Vert _{1-var,\left[ s,t\right] }\leq \left\Vert x\right\Vert
_{1-var,\left[ s,t\right] }$. By Taylor expansions and that $f_{i}$ is $%
Lip\left( \gamma _{i}\right) $ for $\gamma _{i}>1$, $i=1,\dots ,d$, 
\begin{equation*}
\left\vert \Gamma ^{t_{j}^{n},t_{j+1}^{n}}\right\vert \leq C\left\Vert
x\right\Vert _{1-var,\left[ t_{j}^{n},t_{j+1}^{n}\right] }^{2}.
\end{equation*}%
Let $n\rightarrow \infty $ in $\left( \ref{inner estimate Gamma s,t}\right) $%
,%
\begin{equation}
\left\vert \Gamma ^{s,t}\right\vert \leq C\omega \left( s,t\right) ^{\theta
}.  \label{inner1}
\end{equation}%
Combining $\left( \ref{inner1}\right) $, Lemma \ref{Lemma A} and $\left\Vert
x^{s,t,i}\right\Vert _{1-var,\left[ s,t\right] }\leq C\omega \left(
s,t\right) ^{\frac{1}{p_{i}}}$, when $\omega \left( s,t\right) \leq 1$,%
\begin{equation}
\left\vert y_{t}-y_{s}-\sum_{w\in W,0<\left\vert w\right\vert \leq
1}F^{w}\left( y_{s}\right) \left( S_{\Pi }\left( x\right) _{s,t},w\right)
\right\vert \leq C\omega \left( s,t\right) ^{\theta }
\label{bounded variation estimate}
\end{equation}%
where the constant $C$ in $\left( \ref{bounded variation estimate}\right) $
is continuous with respect to $\Pi =\left( p_{1},\cdots ,p_{d}\right) $ in a
neighborhood of $\Pi $.

So far we assumed that $x$ is continuous and of bounded variation. Suppose $%
X $ is a weakly geometric $\Pi $-rough path for $\Pi =\left( p_{1},\cdots
,p_{d}\right) $. Then $X$ is a geometric $\Pi ^{\prime }$-rough path for $%
\Pi ^{\prime }=\left( p_{1}^{\prime },\cdots ,p_{d}^{\prime }\right) $, $%
p_{i}^{\prime }>p_{i}$ \cite[Lemma 3.5]{boedihardjo2019isomorphism}. By
applying universal limit theorem \cite[Theorem 4.3]{gyurko2016differential},
using the continuity of the control $\omega $ with respect to $\Pi $ \cite[%
Lemma 5.13]{friz2010multidimensional} and the continuity of the constant $C$
in $\left( \ref{bounded variation estimate}\right) $ with respect to $\Pi $,
when $\omega \left( s,t\right) \leq 1$,%
\begin{equation}
\left\vert y_{t}-y_{s}-\sum_{w\in W,0<\left\vert w\right\vert \leq
1}F^{w}\left( y_{s}\right) \left( X_{s,t},w\right) \right\vert \leq C\omega
\left( s,t\right) ^{\theta }.  \label{inner 3}
\end{equation}

When $\omega \left( s,t\right) >1$, divide $\left[ s,t\right] =\left[
t_{0},t_{1}\right] \cup \left[ t_{1},t_{2}\right] \cup \cdots \cup \left[
t_{n-1},t_{n}\right] $ such that $\omega \left( t_{k},t_{k+1}\right) =1$ for 
$k=0,\dots ,n-2$ and $\omega \left( t_{n-1},t_{n}\right) \leq 1$. By
super-additivity of $\omega $, $n-1<\omega \left( s,t\right) $ and $n<\omega
\left( s,t\right) +1$. Since $\omega \left( t_{k},t_{k+1}\right) \leq 1$,
based on $\left( \ref{inner 3}\right) $ and Lemma \ref{Lemma bound on Fw
inhomogeneous},%
\begin{equation*}
\left\vert y_{t_{k+1}}-y_{t_{k}}\right\vert \leq C\omega \left(
t_{k},t_{k+1}\right) ^{\frac{1}{p_{\max }}}\leq C\text{ for }k=0,1,\dots ,n-1%
\text{.}
\end{equation*}%
Hence,%
\begin{equation}
\left\vert y_{t}-y_{s}\right\vert \leq \sum_{k=0}^{n-1}\left\vert
y_{t_{k+1}}-y_{t_{k}}\right\vert \leq Cn<C\left( \omega \left( s,t\right)
+1\right) <2C\omega \left( s,t\right) \text{.}  \label{inner 4}
\end{equation}%
On the other hand, based on Lemma \ref{Lemma bound on Fw inhomogeneous} and $%
\left\vert \left( X_{s,t},w\right) \right\vert \leq \omega \left( s,t\right)
^{\left\vert w\right\vert }$, $\left\vert w\right\vert \leq 1$,%
\begin{equation}
\left\vert \sum_{w\in W,0<\left\vert w\right\vert \leq 1}F^{w}\left(
y_{s}\right) \left( X_{s,t},w\right) \right\vert \leq C\omega \left(
s,t\right) .  \label{inner 5}
\end{equation}%
Combining $\left( \ref{inner 4}\right) $ and $\left( \ref{inner 5}\right) $,
when $\omega \left( s,t\right) >1$, we have%
\begin{equation*}
\left\vert y_{t}-y_{s}-\sum_{w\in W,0<\left\vert w\right\vert \leq
1}F^{w}\left( y_{s}\right) \left( X_{s,t},w\right) \right\vert \leq C\omega
\left( s,t\right) .
\end{equation*}
\end{proof}

In the following, we suppose there exists $p\geq 1$ such that $X$ is a
weakly geometric $\Pi $-rough path for $\Pi =\left( p_{1},\cdots
,p_{d}\right) $ where $p_{i}=pk_{i}^{-1}$ for some $k_{i}\in \left\{
1,2,\cdots ,\left[ p\right] \right\} $. Suppose $f=\left( f_{1},\cdots
,f_{d}\right) $ where $f_{i}:%
\mathbb{R}
^{n}\rightarrow 
\mathbb{R}
^{n}$ is $Lip\left( \gamma _{i}\right) $ for some $\gamma _{i}>p_{\max
}\left( 1-k_{i}p^{-1}\right) +1$. Let $y:\left[ 0,T\right] \rightarrow 
\mathbb{R}
^{n}$ denote the unique solution of the rough differential equation%
\begin{equation*}
dy_{t}=f\left( y_{t}\right) dX_{t},y_{0}=\xi .
\end{equation*}%
For $0\leq s<t\leq T$, \ suppose $x^{s,t}:\left[ s,t\right] \rightarrow 
\mathbb{R}
^{d}$ is a geodesic associated with $X_{s,t}$. Let $y^{s,t}$ denote the
unique solution of the ODE%
\begin{equation*}
dy_{u}^{s,t}=f\left( y_{u}^{s,t}\right) dx_{u}^{s,t},y_{s}^{s,t}=y_{s}\text{.%
}
\end{equation*}%
Denote $\gamma :=\min_{i=1}^{d}\left\{ \left( \gamma _{i}-1\right) pp_{\max
}^{-1}+k_{i}\right\} $ and denote $N:=\lfloor \gamma \rfloor $. Then $N\geq %
\left[ p\right] $.

\begin{notation}
Suppose $p_{i}=pk_{i}^{-1}$ for $k_{i}\in \left\{ 1,\dots ,\left[ p\right]
\right\} $. For $w\in W$, denote 
\begin{equation*}
\left\Vert w\right\Vert :=p\left\vert w\right\vert .
\end{equation*}%
Then for $w=i_{1}\cdots i_{m}$, $\left\Vert w\right\Vert =k_{i_{1}}+\cdots
+k_{i_{m}}$.
\end{notation}

\begin{lemma}
\label{Lemma bound on Fw}For $w\in W$, $\left\Vert w\right\Vert \leq N$, $%
\left\vert F^{w}\left( y\right) -F^{w}\left( x\right) \right\vert \leq
\left\Vert w\right\Vert !\left\vert y-x\right\vert $ for every $x,y\in 
\mathbb{R}
^{n}$ and for $i\in \left\{ 1,\cdots ,d\right\} $, $\sup_{y\in 
\mathbb{R}
^{n}}\left\vert F^{iw}\left( y\right) \right\vert \leq \left\Vert
w\right\Vert !$
\end{lemma}

\begin{proof}
Since $\gamma :=\min_{i=1}^{d}\left\{ \left( \gamma _{i}-1\right) pp_{\max
}^{-1}+k_{i}\right\} $ and $N:=\lfloor \gamma \rfloor <\gamma $, we have 
\begin{equation*}
\gamma _{i}>\left( N-k_{i}\right) p^{-1}p_{\max }+1=p_{\max }\left(
Np^{-1}-p_{i}^{-1}\right) +1.
\end{equation*}%
Based on Lemma \ref{Lemma indices} (with $\alpha =Np^{-1}$), when $%
\left\Vert w\right\Vert \leq N$, $F^{w}$ is $Lip\left( \eta \right) $ for
some $\eta >1$. Suppose $w=i_{1}\cdots i_{m}$. Consider $F^{w}=\left(
f_{i_{1}}\circ f_{i_{2}}\circ \cdots \circ f_{i_{m}}\right) \left( I\right) $%
. Since the differential operator $f_{i_{j}}D$ can choose one from $\left\{
f_{i_{j+1}},\cdots ,f_{i_{m}}\right\} $ for $j=1,2,\dots ,m-1$, there are $%
\left( m-1\right) !$ functions in $F^{w}=\left( f_{i_{1}}\circ
f_{i_{2}}\circ \cdots \circ f_{i_{m}}\right) \left( I\right) $. For $DF^{w}$%
, as $D$ choosing one from $\left\{ f_{i_{1}},\cdots ,f_{i_{m}}\right\} $,
there is another factor of $m$, so total $m!$ functions in $DF^{w}$. Since $%
\left\Vert f_{i}\right\Vert _{Lip\left( \gamma _{i}\right) }=1$, each of
these functions is bounded by $1$. Hence, $\left\Vert DF^{w}\right\Vert
_{\infty }\leq m!\leq \left\Vert w\right\Vert !$ and $\left\vert F^{w}\left(
y\right) -F^{w}\left( x\right) \right\vert \leq \left\Vert DF^{w}\right\Vert
_{\infty }\left\vert y-x\right\vert \leq \left\Vert w\right\Vert !\left\vert
y-x\right\vert $. For the second bound, $\sup_{y\in 
\mathbb{R}
^{n}}\left\vert F^{iw}\left( y\right) \right\vert =\sup_{y\in 
\mathbb{R}
^{n}}\left\vert DF^{w}\left( y\right) f_{i}\left( y\right) \right\vert \leq
\left\Vert DF^{w}\right\Vert _{\infty }\left\Vert f_{i}\right\Vert _{\infty
}\leq \left\Vert w\right\Vert !$.
\end{proof}

\begin{lemma}
Suppose $\omega \left( s,t\right) \leq 1$. For $w\in W$, $\left\Vert
w\right\Vert =N-\left[ p\right] ,\cdots ,N$,%
\begin{eqnarray}
&&\left\vert F^{w}\left( y_{t}\right) -F^{w}\left( y_{s}\right) -\sum_{1\leq
\left\Vert l\right\Vert \leq N-\left\Vert w\right\Vert }F^{lw}\left(
y_{s}\right) \left( X_{s,t},l\right) \right\vert  \label{estimate1} \\
&\leq &CN!\frac{\omega \left( s,t\right) ^{\frac{N-\left\Vert w\right\Vert +1%
}{p}}}{\left( \frac{N-\left\Vert w\right\Vert +1}{p}\right) !}  \notag
\end{eqnarray}%
When $\left\Vert w\right\Vert =0,\dots ,N-\left[ p\right] -1$,%
\begin{eqnarray}
&&\left\vert F^{w}\left( y_{t}\right) -F^{w}\left( y_{s}\right) -\sum_{1\leq
\left\Vert l\right\Vert \leq \left[ p\right] }F^{lw}\left( y_{s}\right)
\left( X_{s,t},l\right) \right\vert  \label{estimate2} \\
&\leq &C\left( \left\Vert w\right\Vert +\left[ p\right] \right) !\omega
\left( s,t\right) ^{\frac{\left[ p\right] +1}{p}}  \notag
\end{eqnarray}
\end{lemma}

\begin{proof}
When $\left\Vert w\right\Vert =N$, based on Lemma \ref{Lemma bound on Fw},
Lemma \ref{Lemma Euler estimate} and $\omega \left( s,t\right) \leq 1$, 
\begin{equation*}
\left\vert F^{w}\left( y_{t}\right) -F^{w}\left( y_{s}\right) \right\vert
\leq N!\left\vert y_{t}-y_{s}\right\vert \leq CN!\omega \left( s,t\right) ^{%
\frac{1}{p_{\max }}}\leq CN!\omega \left( s,t\right) ^{\frac{1}{p}}.
\end{equation*}%
Based on Lemma \ref{Lemma bound on Fw}, Lemma \ref{Lemma Euler estimate},
Lemma \ref{Lemma A}, $\left\Vert x^{s,t,i}\right\Vert _{1-var,\left[ s,t%
\right] }\leq C\omega \left( s,t\right) ^{\frac{1}{p_{i}}}$, $S_{\Pi }\left(
x^{s,t}\right) _{s,t}=X_{s,t}$ and $\theta \geq \left( \left[ p\right]
+1\right) p^{-1}$, since $\omega \left( s,t\right) \leq 1$,%
\begin{equation}
\left\vert F^{w}\left( y_{t}\right) -F^{w}\left( y_{t}^{s,t}\right)
\right\vert \leq \left\Vert w\right\Vert !\left\vert
y_{t}-y_{t}^{s,t}\right\vert \leq C\left\Vert w\right\Vert !\omega \left(
s,t\right) ^{\frac{\left[ p\right] +1}{p}}.  \label{inner6}
\end{equation}%
On the other hand, since $\left\Vert x^{s,t,i}\right\Vert _{1-var,\left[ s,t%
\right] }\leq C\omega \left( s,t\right) ^{\frac{1}{p_{i}}}$ and $\omega
\left( s,t\right) \leq 1$, 
\begin{equation}
\max_{u\in \left[ s,t\right] }\left\vert y_{u}^{s,t}-y_{s}\right\vert \leq
\sum_{i=1}^{d}\left\Vert x^{s,t,i}\right\Vert _{1-var,\left[ s,t\right]
}\leq C\omega \left( s,t\right) ^{\frac{1}{p_{\max }}}\leq C\omega \left(
s,t\right) ^{\frac{1}{p}}.  \label{inner estimate}
\end{equation}%
Then when $\left\Vert w\right\Vert =N-\left[ p\right] ,\cdots ,N-1$, based
on the fundamental theorem of calculus, Lemma \ref{Lemma bound on Fw}, $%
\left( \ref{inner estimate}\right) $ and $\left\Vert x^{s,t,i}\right\Vert
_{1-var,\left[ s,t\right] }\leq C\omega \left( s,t\right) ^{\frac{1}{p_{i}}}$%
, since $\omega \left( s,t\right) \leq 1$, 
\begin{eqnarray}
&&\left\vert F^{w}\left( y_{t}^{s,t}\right) -F^{w}\left( y_{s}\right)
-\sum_{0<\left\Vert l\right\Vert \leq N-\left\Vert w\right\Vert
}F^{lw}\left( y_{s}\right) \left( X_{s,t},l\right) \right\vert
\label{inner7} \\
&\leq &\sum_{\left\Vert i_{1}\cdots i_{k}w\right\Vert
=N}\;\;\idotsint\limits_{s<u_{1}<\cdots <u_{k}<t}\left\vert F^{i_{1}\cdots
i_{k}w}\left( y_{u}^{s,t}\right) -F^{i_{1}\cdots i_{k}w}\left( y_{s}\right)
\right\vert \left\vert dx_{u_{1}}^{s,t,i_{1}}\right\vert \cdots \left\vert
dx_{u_{k}}^{s,t,i_{k}}\right\vert  \notag \\
&&+\sum_{\substack{ \left\Vert i_{2}\cdots i_{k}w\right\Vert <N  \\ %
\left\Vert i_{1}\cdots i_{k}w\right\Vert >N}}\;\;\idotsint\limits_{s<u_{1}<%
\cdots <u_{k}<t}\left\vert F^{i_{1}\cdots i_{k}w}\left( y_{u}^{s,t}\right)
\right\vert \left\vert dx_{u_{1}}^{s,t,i_{1}}\right\vert \cdots \left\vert
dx_{u_{k}}^{s,t,i_{k}}\right\vert  \notag \\
&\leq &CN!\omega \left( s,t\right) ^{\frac{N-\left\Vert w\right\Vert +1}{p}}.
\notag
\end{eqnarray}%
Combining $\left( \ref{inner6}\right) $ and $\left( \ref{inner7}\right) $,
the estimate $\left( \ref{estimate1}\right) $ holds. When $\left\Vert
w\right\Vert =0,\dots ,N-\left[ p\right] -1$, similar arguments apply and $%
\left( \ref{estimate2}\right) $ holds.
\end{proof}

\begin{proof}[Proof of Theorem \protect\ref{Theorem Taylor estimate}]
By an inhomogeneous analogue of \cite[Theorem 2.2.1]{lyons1998differential},
for $s<t$ and $l\in W$, $\left\Vert l\right\Vert =1,2,\dots $%
\begin{equation}
\left\vert \left( X_{s,t},l\right) \right\vert \leq \beta ^{\left\Vert
l\right\Vert -1}\frac{\omega \left( s,t\right) ^{\frac{\left\Vert
l\right\Vert }{p}}}{\left( \frac{\left\Vert l\right\Vert }{p}\right) !}
\label{factorial decay}
\end{equation}

Firstly assume $\omega \left( s,t\right) \leq 1$. Denote $%
Y_{t}^{w}:=F^{w}\left( y_{t}\right) $. Denote $\widetilde{\omega }:=\beta
^{p}d^{p}\omega $. Inductive hypothesis: suppose for $s<t$, $\omega \left(
s,t\right) \leq 1$ and $w\in W$, $\left\Vert w\right\Vert =n+1,\cdots ,N$,
we have 
\begin{equation}
\left\vert Y_{t}^{w}-Y_{s}^{w}-\sum_{1\leq \left\Vert l\right\Vert \leq
N-\left\Vert w\right\Vert }Y_{s}^{lw}\left( X_{s,t},l\right) \right\vert
\leq CN!\frac{\widetilde{\omega }\left( s,t\right) ^{\frac{N-\left\Vert
w\right\Vert +1}{p}}}{\beta \left( \frac{N-\left\Vert w\right\Vert +1}{p}%
\right) !}  \label{inductive hypothesis}
\end{equation}%
which holds when $n=N-\left[ p\right] -1$ based on $\left( \ref{estimate1}%
\right) $.

Fix $w\in W,\left\Vert w\right\Vert =n$, $n\leq N-\left[ p\right] -1$.
Denote 
\begin{equation*}
L_{s,t}=\sum_{1\leq \left\Vert l\right\Vert \leq N-\left\Vert w\right\Vert
}Y_{s}^{lw}\left( X_{s,t},l\right)
\end{equation*}%
Then based on $\left( \ref{estimate2}\right) $, Lemma \ref{Lemma bound on Fw}
and $\left( \ref{factorial decay}\right) $, since $\omega \left( s,t\right)
\leq 1$,%
\begin{eqnarray*}
&&\left\vert Y_{t}^{w}-Y_{s}^{w}-L_{s,t}\right\vert \\
&\leq &\left\vert Y_{t}^{w}-Y_{s}^{w}-\sum_{1\leq \left\Vert l\right\Vert
\leq \left[ p\right] }Y_{s}^{lw}\left( X_{s,t},l\right) \right\vert
+\left\vert \sum_{\left[ p\right] +1\leq \left\Vert l\right\Vert \leq
N-\left\Vert w\right\Vert }Y_{s}^{lw}\left( X_{s,t},l\right) \right\vert \\
&\leq &CN!\omega \left( s,t\right) ^{\frac{\left[ p\right] +1}{p}}.
\end{eqnarray*}%
Hence,%
\begin{equation*}
Y_{t}^{w}-Y_{s}^{w}=\lim_{\left\vert D\right\vert \rightarrow 0,D\subset 
\left[ s,t\right] }\sum_{k,t_{k}\in D}L_{t_{k},t_{k+1}}.
\end{equation*}

For $s<u<t$, by using that $\#\left\{ w\in W|\left\Vert w\right\Vert
=n\right\} \leq d^{n}$, the inductive hypothesis $\left( \ref{inductive
hypothesis}\right) $, the factorial decay $\left( \ref{factorial decay}%
\right) $ and the neo-classical inequality \cite{hara2010fractional}, 
\begin{eqnarray*}
&&\left\vert L_{s,u}+L_{u,t}-L_{s,t}\right\vert \\
&=&\left\vert \sum_{1\leq \left\Vert l\right\Vert \leq N-\left\Vert
w\right\Vert }\left( Y_{u}^{lw}-Y_{s}^{lw}-\sum_{1\leq \left\Vert
l_{1}\right\Vert \leq N-\left\Vert lw\right\Vert }Y_{s}^{l_{1}lw}\left(
X_{s,u},l_{1}\right) \right) \left( X_{u,t},l\right) \right\vert \\
&\leq &\sum_{n=1}^{N-\left\Vert w\right\Vert }Cd^{n}N!\frac{\widetilde{%
\omega }\left( s,u\right) ^{\frac{N-n-\left\Vert w\right\Vert +1}{p}}}{\beta
\left( \frac{N-n-\left\Vert w\right\Vert +1}{p}\right) !}\frac{\beta
^{n-1}\omega \left( u,t\right) ^{\frac{n}{p}}}{\left( \frac{n}{p}\right) !}
\\
&\leq &\sum_{n=1}^{N-\left\Vert w\right\Vert }CN!\frac{\widetilde{\omega }%
\left( s,u\right) ^{\frac{N-n-\left\Vert w\right\Vert +1}{p}}}{\beta \left( 
\frac{N-n-\left\Vert w\right\Vert +1}{p}\right) !}\frac{\widetilde{\omega }%
\left( u,t\right) ^{\frac{n}{p}}}{\beta \left( \frac{n}{p}\right) !}\text{ \
\ \ \ (since }\widetilde{\omega }:=\beta ^{p}d^{p}\omega \text{)} \\
&\leq &CN!\frac{p\,\widetilde{\omega }\left( s,t\right) ^{\frac{N-\left\Vert
w\right\Vert +1}{p}}}{\beta ^{2}\left( \frac{N-\left\Vert w\right\Vert +1}{p}%
\right) !}\text{ \ \ \ \ \ \ \ \ \ (based on neo-classical inequality)}
\end{eqnarray*}%
Then by sequentially removing partitions points as in \cite[Theorem 2.2.1]%
{lyons1998differential}, $\left( \ref{inductive hypothesis}\right) $ holds
when $\left\Vert w\right\Vert =n$ and the induction is complete. Let $w$ be
the empty sequence $\varepsilon $. Then $Y_{t}^{\varepsilon }=y_{t}$, $%
\left\Vert \varepsilon \right\Vert =0$ and the estimate $\left( \ref{main
estimate}\right) $ holds when $\omega \left( s,t\right) \leq 1$.

Suppose $\omega \left( s,t\right) >1$. Based on Lemma \ref{Lemma bound on Fw}%
, when $\left\Vert w\right\Vert =j$, $\sup_{y\in 
\mathbb{R}
^{n}}\left\vert F^{w}\left( y\right) \right\vert \leq \left( j-1\right) !$.
Combined with $\#\left\{ w\in W|\left\Vert w\right\Vert =j\right\} \leq d^{j}
$ and the factorial decay at $\left( \ref{factorial decay}\right) $,%
\begin{equation*}
\left\vert \sum_{w\in W,\left\Vert w\right\Vert =j}F^{w}\left( y_{s}\right)
\left( X_{s,t},w\right) \right\vert \leq \left( j-1\right) !d^{j}\beta ^{j-1}%
\frac{\omega \left( s,t\right) ^{\frac{j}{p}}}{\left( \frac{j}{p}\right) !}
\end{equation*}%
For $j=1,2,\dots $, denote 
\begin{equation*}
a_{j}:=\frac{\left( j-1\right) !}{\left( \frac{j}{p}\right) !}\beta ^{j-1}
\end{equation*}%
Since $\beta >1$ only depends on $p$, by Kershaw's inequality \cite[$\left(
1.3\right) $]{kershaw1983some}, for $\lambda =2^{-1}\left( \beta
^{-1}+1\right) \in \left( 0,1\right) $ there exists $C_{p}>0$ such that $%
a_{j}\leq C_{p}\lambda ^{N+1-j}a_{N+1}$ for $j=1,\dots ,N$ and $N=1,2,\dots $%
. Then%
\begin{equation}
\sum_{j=1}^{N}a_{j}\leq C_{p}\sum_{j=1}^{N}\lambda ^{N+1-j}a_{N+1}\leq
C_{p}^{\prime }a_{N+1}  \label{inner8}
\end{equation}%
On the other hand, since $\lim_{N\rightarrow \infty }a_{N+1}=\infty $ and $%
a_{N+1}>0$, there exists $C_{p}>0$ such that 
\begin{equation}
1\leq C_{p}a_{N+1}  \label{inner9}
\end{equation}%
Based on Lemma \ref{Lemma Euler estimate} and Lemma \ref{Lemma bound on Fw
inhomogeneous}, $\left\vert y_{t}-y_{s}\right\vert \leq C\omega \left(
s,t\right) $ when $\omega \left( s,t\right) >1$. Since $d\geq 1,\omega
\left( s,t\right) ^{\frac{1}{p}}>1$ and $\frac{N+1}{p}\geq \frac{\left[ p%
\right] +1}{p}>1$, combining $\left( \ref{inner8}\right) $ and $\left( \ref%
{inner9}\right) $, we have%
\begin{eqnarray*}
&&\left\vert y_{t}-y_{s}-\sum_{w\in W,1\leq \left\Vert w\right\Vert \leq
N}F^{w}\left( y_{s}\right) \left( X_{s,t},w\right) \right\vert  \\
&\leq &C\omega \left( s,t\right) +\sum_{j=1}^{N}\left( j-1\right)
!d^{j}\beta ^{j-1}\frac{\omega \left( s,t\right) ^{\frac{j}{p}}}{\left( 
\frac{j}{p}\right) !} \\
&\leq &CN!d^{N+1}\beta ^{N}\frac{\omega \left( s,t\right) ^{\frac{N+1}{p}}}{%
\left( \frac{N+1}{p}\right) !}.
\end{eqnarray*}
\end{proof}

\bibliographystyle{unsrt}
\bibliography{acompat,roughpath}

\end{document}